\documentclass[a4paper,12pt]{article}
\pdfoutput=1
\usepackage[cp1251]{inputenc}
\usepackage[T2A]{fontenc}
\usepackage[english]{babel}
\usepackage[intlimits]{amsmath}
\usepackage{amsthm}
\usepackage{bm}
\usepackage{amsfonts,amssymb}

\def\leq{\leqslant}
\def\geq{\geqslant}
\DeclareMathOperator{\meas}{meas}
\theoremstyle{plain}
\newtheorem{remark}{Remark}

\newtheorem{prop}{Proposition}
\newtheorem{theorem}{Theorem}

\begin{document}
\title{On the constancy of the extremal function \\ in the embedding theorem of fractional order}

\author{Nikita Ustinov\footnote{St.Petersburg State University, 7/9 Universitetskaya Emb., St.Petersburg 199034, Russia. E-mail: {ustinns@yandex.ru}.  Supported by RFBR grant 20-01-00630A}}

\date{}

\maketitle

{\footnotesize
\noindent
{\bf Abstract.} We consider the problem of the minimizer constancy in the fractional embedding theorem $\mathcal{H}^s(\Omega) \hookrightarrow L_q(\Omega)$ for a bounded Lipschitz domain $\Omega,$ depending on the domain size. For the family of domains $\varepsilon \Omega,$ we prove that for small dilation coefficients $\varepsilon$ a unique minimizer is constant, whereas for large~$\varepsilon$ a constant function is not even a local minimizer. We also discuss whether a constant function is a global minimizer if it is a local one.
\medskip


\medskip
\noindent
{\bf 2010 Mathematics Subject Classfication:} 35R11, 47J30, 34A08, 34A05}

\section{Introduction}
Let $n \geq 1$ and let $\Omega \subset \mathbb{R}^n$ be a bounded domain with Lipschitz boundary. Assume that $s \in (0, 1),$ and
\begin{equation*}
q \in
\begin{cases}
[1, 2^*_{s}]& \text{for $n \geq 2$ \  or \ $n=1,$ $s < \frac{1}{2};$} \\
[1, \infty) & \text{for $n= 1,$ $s = \frac{1}{2};$} \\
[1, \infty] & \text{for $n= 1,$ $s > \frac{1}{2}.$}
\end{cases}
\end{equation*}
($2^*_{s} := 2n/(n-2s)$ stands for the critical Sobolev exponent). We consider the continuous embedding of the Sobolev--Slobodetskii space $\mathcal{H}^s(\Omega)$ (see \cite{slob} and \cite[\S 2.3.3]{Triebel}) into $L_q(\Omega):$
\begin{equation}
\label{emb_theorem_main}
\inf_{u \in \mathcal{H}^s(\Omega)} \mathcal{I}^{\, \Omega}_{s, q}[u] := \inf_{u \in \mathcal{H}^s(\Omega)} \frac{\| u \|^2_{\mathcal{H}^s(\Omega)}}{\| u \|^2_{L_q(\Omega)}} > 0.
\end{equation}
This embedding is compact for $q \in [1, 2^*_{s});$ therefore, the extremal (minimizer of the functional~$\mathcal{I}^{\, \Omega}_{s, q}[u]$) exists, and the sharp constant in \eqref{emb_theorem_main} is attained. Moreover, in \cite{Ustinov_Neum} it was shown that in case $q = 2^*_{s}$ the extremal in \eqref{emb_theorem_main} exists in $\mathcal{C}^2$--smooth domain $\Omega$ for $n \geq 3,$ $2s > 1.$

Obviously, the properties of extremals in \eqref{emb_theorem_main} depend on the shape of the domain $\Omega,$ on its size and on the norm in $\mathcal{H}^s(\Omega).$ We define 
\begin{equation}
\label{H_s_norm}
\| u \|^2_{\mathcal{H}^s(\Omega)} := \langle (-\Delta)_{Sp}^s u, u \rangle  + \| u \|^2_{L_2(\Omega)},
\end{equation}
where the quadratic form $\langle (-\Delta)_{Sp}^s u, u \rangle$ is defined by
\begin{equation}
\label{seminorm_Sp}
\langle(-\Delta)_{Sp}^s u,u\rangle := \sum\limits_{j=1}^{\infty} \lambda_j^s \cdot ( u,\phi_j )_{L_2(\Omega)}^2.
\end{equation}
Here $\lambda_j$ and $\phi_j$ are eigenvalues and eigenfunctions (orthonormalized in $L_2(\Omega)$) of the Neumann Laplacian in $\Omega,$ respectively (we assume that $\lambda_0 = 0$ for $\phi_0 = C$).

The operator $(-\Delta)_{Sp}^s,$ generated by \eqref{seminorm_Sp}, is the $s$-th power of the conventional Neumann Laplacian in the sense of spectral theory. It is called the \textit{Spectral Neumann fractional Laplacian}.

The domain of the form $\langle(-\Delta)_{Sp}^s u,u\rangle$ coincides with the space $\mathcal{H}^s(\Omega),$ and the norm~\eqref{H_s_norm} is equivalent to the standard norm in $\mathcal{H}^s(\Omega)$ (the proof of these facts is similar to the proof of Lemma 1 in \cite{MN2}).

In what follows, we assume that $\meas (\Omega) = 1,$ and consider the embedding \eqref{emb_theorem_main} for the set of domains $\Omega_{\varepsilon} := \{ \varepsilon x \  | \ x \in \Omega \}.$ We put $u_{\varepsilon}(y) := u(\varepsilon^{-1}y),$ then
\begin{equation}
\label{fu_s_omega_varepsilon}
\mathcal{I}^{\varepsilon}_{s, q}[u]
:=
\frac{\mathcal{I}^{\,\Omega_{\varepsilon}}_{s, q}[u_{\varepsilon}]}{\varepsilon^{n-2s - \frac{2n}{q}}}
=
\frac{ \varepsilon^{n-2s}  \langle (-\Delta)_{Sp}^s u, u \rangle  + \varepsilon^{n} \| u \|^2_{L_2(\Omega)}}{\varepsilon^{n-2s - \frac{2n}{q}} \varepsilon^{\frac{2n}{q}} \| u \|^2_{L_q(\Omega)}}
=
 \frac{\langle (-\Delta)_{Sp}^s u, u \rangle  + \varepsilon^{2s} \| u \|^2_{L_2(\Omega)}}{\| u \|^2_{L_q(\Omega)}}.
\end{equation}
The functionals $\mathcal{I}^{\varepsilon}_{s, q}[u]$ and $\mathcal{I}^{\, \Omega}_{s, q}[u]$ are invariant under the multiplications of the function $u$ by a constant. Therefore we can assume that $\| u \|^2_{L_q(\Omega)} = 1.$

In the local case $s = 1$ the following facts about the minimizer in \eqref{fu_s_omega_varepsilon} are known (recall that $\mathcal{H}^1(\Omega)$ coincides with the Sobolev space $W^1_2(\Omega)$ and the norm \eqref{H_s_norm} coincides with the standard one):
\begin{enumerate}
\item If $q \in [1,2],$ then for any $\varepsilon > 0$ the constant function $\mathbf{1}$ (identically equal to one) is a unique minimizer of the functional $\mathcal{I}^{\varepsilon}_{1, q}[u];$
\item If $q \in (2,2^*_{1}]$ and $\lambda_1$ is the first nonzero eigenvalue of the Neumann Laplacian in $\Omega,$ then
\begin{itemize}
\item for $\varepsilon >\varepsilon_1(q) := \sqrt{\lambda_1 / (q-2)}$ the constant function $\mathbf{1}$ is not a \textit{local} minimizer of the functional $\mathcal{I}^{\varepsilon}_{1, q}[u]$ (see \cite[Proposition 3.1]{Naz_Sch}). Obviously, in this case the constant function~$\mathbf{1}$ is not a \textit{global} minimizer;
\item for $\varepsilon < \varepsilon_1(q)$ the constant function $\mathbf{1}$ gives a \textit{local} minimum for the functional $\mathcal{I}^{\varepsilon}_{1, q}[u]$ (see \cite[Proposition 3.2]{Naz_Sch});
\item there exists such $\mathcal{E}_1(q) > 0,$ that for all $\varepsilon \leq \mathcal{E}_1(q)$ the constant function $\mathbf{1}$ gives a \textit{global} minimum for the functional $\mathcal{I}^{\varepsilon}_{1, q}[u],$ whereas for all $\varepsilon > \mathcal{E}_1(q)$ the constant function $\mathbf{1}$ is not a \textit{global} minimizer for the functional $\mathcal{I}^{\varepsilon}_{1, q}[u]$ (see \cite[Propositions 3.5--3.7]{Naz_Sch});
\item the function $\mathcal{E}_1(q)$ is continuous and monotonically decreasing (see \cite[Theorem 3.8]{Naz_Sch}); for $n = 1$ one has $\mathcal{E}_1(q) = \varepsilon_1(q)$ (see \cite{Naz2}, \cite{Naz1}); for $n \geq 2$ there exist both examples of domains with $\mathcal{E}_1(q) < \varepsilon_1(q)$ (see \cite[Corollary 3.4]{Naz_Sch}, \cite[(3.6)]{Naz_Sch}) and examples of domains with $\mathcal{E}_1(q) = \varepsilon_1(q)$ (see \cite[Theorem 3.1]{Sch}).
\end{itemize}
\end{enumerate}
In this paper, we study similar questions for the functional~\eqref{fu_s_omega_varepsilon} in the nonlocal case, provided that its minimizer exists. Standard variational argument shows that for $q < \infty$ the minimizer (if it exists) solves the following problem (up to multiplication by a constant)
 \begin{equation}
 \label{neumann_pr_s}
(-\Delta)_{Sp}^s u + u = |u|^{q - 2} u, \quad u\in \mathcal{H}^s(\Omega).
\end{equation}

\begin{remark}
From \eqref{seminorm_Sp} it follows that $(-\Delta)_{Sp}^s \mathbf{1} = 0$ and $\mathcal{I}^{\varepsilon}_{s, q}[\mathbf{1}] = \varepsilon^{2s}.$Thus, the function $u=\mathbf{1}$ is always a solution to the problem \eqref{neumann_pr_s}.
\end{remark}

\begin{remark}
For $q \in [1, 2]$ the H{\"o}lder inequality provides $\| u \|^2_{L_q(\Omega)} \leq \| u \|^2_{L_2(\Omega)},$ therefore $\mathcal{I}^{\varepsilon}_{s, q}[u] \geq  \varepsilon^{2s} = \mathcal{I}^{\varepsilon}_{s, q}[\mathbf{1}].$ Since the H{\"o}lder inequality is strict for any non-constant function $u,$ the constant function is the only minimizer of the functional $\mathcal{I}^{\varepsilon}_{s, q}[u]$ for any $s \in (0, 1].$
\end{remark}

\begin{remark}
We can assume that the minimizer of the functional \eqref{fu_s_omega_varepsilon} is non-negative. Indeed, the substitution $u \to |u|$ decreases the value of the functional $\mathcal{I}^{\varepsilon}_{s, q}[u].$ Namely, the norms in $L_2$ and $L_q$ are invariant under such substitution, whereas the quadratic form $\langle(-\Delta)_{Sp}^s u,u\rangle$ decreases for any sign-changing function $u$
\begin{equation*}
 \langle(-\Delta)_{Sp}^s u,u\rangle  > \langle(-\Delta)_{Sp}^s |u|, |u|\rangle,
\end{equation*}
(see the proof of Theorem 3 in \cite{MN4}).
\end{remark}

In what follows, we assume that $q > 2,$ and all minimizers of the functionals $\mathcal{I}^{\varepsilon}_{s, q}[u]$ are non-negative.

\medskip

We use the letter $C$ to denote various positive constants, the exact meaning of which is not important for us. To indicate that $C$ depends on some other parameters, we write $C(\dots).$ We also omit the dependence of $C$ on the domain $\Omega.$

\section{The main result}
Let us define an auxiliary functional
\begin{equation*}
\mathcal{J}^{\varepsilon}_{s, q}[u] 
:= 
\| u \|^2_{L_q(\Omega)} \cdot \left( \mathcal{I}^{\varepsilon}_{s, q}[u] - \mathcal{I}^{\varepsilon}_{s, q}[\mathbf{1}] \right)
=
\langle (-\Delta)_{Sp}^s u, u \rangle  + \varepsilon^{2s} \| u \|^2_{L_2(\Omega)} - \varepsilon^{2s} \| u \|^2_{L_q(\Omega)},
\end{equation*}
it follows from the construction that the constant function $\mathbf{1}$ is a minimizer of the functional~$\mathcal{I}^{\varepsilon}_{s, q}[u]$ if and only if the functional $\mathcal{J}^{\varepsilon}_{s, q}[u]$ is non-negative (obviously,  $\mathcal{J}^{\varepsilon}_{s, q}[\mathbf{1}] = 0$).

Let us investigate the nonnegativity of the functional $\mathcal{J}^{\varepsilon}_{s, q}[u].$ Its first differential is
\begin{equation*}
\mathbf{D} \mathcal{J}^{\varepsilon}_{s, q}[u; h]
=
2 \Bigl( \langle (-\Delta)_{Sp}^s u, h \rangle  + \varepsilon^{2s} \int\limits_{\Omega} uh \, dx - \varepsilon^{2s} \| u \|^{2 - q}_{L_q(\Omega)} \int\limits_{\Omega} u^{q - 1} h \, dx \Bigr).
\end{equation*}
It is easy to see that $\mathbf{D} \mathcal{J}^{\varepsilon}_{s, q}[\mathbf{1}; h] \equiv 0.$ We also compute the second differential:
\begin{multline}
\label{J_second_derivative}
\mathbf{D}^2 \mathcal{J}^{\varepsilon}_{s, q}[u; h, h]
=
2 \langle (-\Delta)_{Sp}^s h, h \rangle  + 2 \varepsilon^{2s} \| h \|^2_{L_2(\Omega)} 
+  
2(q-2) \varepsilon^{2s} \| u \|^{2(1- q)}_{L_q(\Omega)} \Bigl( \int\limits_{\Omega} u^{q - 1} h \, dx \Bigr)^2 
\\-
2(q-1) \varepsilon^{2s} \| u \|^{2 - q}_{L_q(\Omega)} \int\limits_{\Omega} u^{q - 2} h^2 \, dx.
\end{multline}
We decompose the function $h$ into the sum of terms orthogonal in $L_2(\Omega):$
\begin{equation}
\label{mean_rep}
h(x) = \int\limits_{\Omega} h \, dx + \widehat{h}, \quad \int\limits_{\Omega} \widehat{h} \, dx = 0;
\end{equation}
then for $u = \mathbf{1}$ the second differential \eqref{J_second_derivative} takes the form
\begin{equation*}
\mathbf{D}^2 \mathcal{J}^{\varepsilon}_{s, q}[\mathbf{1}; h, h]
=
2 \langle (-\Delta)_{Sp}^s \widehat{h}, \widehat{h} \rangle  
-
2 (q-2) \varepsilon^{2s} \|\widehat{h}\|^2_{L_2(\Omega)}
=
2\sum\limits_{j=1}^{\infty} \left(\lambda_j^s - (q-2)\varepsilon^{2s} \right) \cdot  (\widehat{h},\phi_j )_{L_2(\Omega)}^2.
\end{equation*}
We denote $\varepsilon_s (q) := (\lambda^s_1/(q-2))^{1/(2s)}$ (obviously, $\varepsilon_s(q) = \varepsilon_1(q)$ for $s=1$).
\begin{theorem}
\label{local_min}
\begin{enumerate}
\item Let $\varepsilon > \varepsilon_s(q).$ Then the constant function $\mathbf{1}$ is not a local minimizer of the functional $\mathcal{I}^{\varepsilon}_{s, q}[u]$ (consequently, the constant function $\mathbf{1}$ is not a global minimizer). This statement remains valid for $q = \infty$ (in this case $\varepsilon_s(q) = 0$).
\item Let $\varepsilon < \varepsilon_s(q).$ Then the constant function $\mathbf{1}$ is a local minimizer of the functional~$\mathcal{I}^{\varepsilon}_{s, q}[u].$
\item Let $\varepsilon = \varepsilon_s(q)$ and $\int\limits_{\Omega} \phi^3_1 \, dx \neq 0.$ Then the constant function $\mathbf{1}$ is not a local minimizer of the functional $\mathcal{I}^{\varepsilon_s}_{s, q}[u].$ Moreover, there exists $\breve{\varepsilon}(q) < \varepsilon_s(q),$ such that for $\varepsilon \in (\breve{\varepsilon}(q), \varepsilon_s(q)]$ the constant function $\mathbf{1}$ is not a global minimizer of the functional $\mathcal{I}^{\varepsilon}_{s, q}[u]$ (although it gives a local minimum).
\end{enumerate}
\end{theorem}
\begin{proof}
1. This statement follows from the inequality $\mathbf{D}^2 \mathcal{J}^{\varepsilon}_{s, q}[\mathbf{1}; \phi_1, \phi_1] < 0.$

\medskip

2. Let $u \in \mathcal{H}^s(\Omega),$ $\| u - \mathbf{1} \|_{\mathcal{H}^s(\Omega)} < \frac{1}{2}.$ We use the decomposition \eqref{mean_rep}: we have $u = c + \widehat{u}$ with $1/2 < c < 3/2,$ and the equality
\begin{equation*}
\mathcal{J}^{\varepsilon}_{s, q}[u] 
=
\langle (-\Delta)_{Sp}^s \widehat{u}, \widehat{u} \rangle  + \varepsilon^{2s} \| \widehat{u} \|^2_{L_2(\Omega)} + c^2\varepsilon^{2s} - c^2 \varepsilon^{2s} \|  \mathbf{1} + \widehat{u}/c \|^2_{L_q(\Omega)}.
\end{equation*}
In \cite[(3.5)]{Naz_Sch} it was shown that
\begin{equation*}
\| \mathbf{1} + \widehat{u}/c \|^2_{L_q(\Omega)} \leq 1 + (q -1) \| \widehat{u}/c \|^2_{L_2(\Omega)} + C(q) \| \widehat{u} \|^{\min(q, 3)}_{L_q(\Omega)},
\end{equation*}
therefore we obtain the inequality
\begin{equation*}
\mathcal{J}^{\varepsilon}_{s, q}[u] 
\geq 
\langle (-\Delta)_{Sp}^s \widehat{u}, \widehat{u} \rangle - (q -2) \varepsilon^{2s} \| \widehat{u} \|^2_{L_2(\Omega)} - C(q) \varepsilon^{2s} \| \widehat{u} \|^{\min(q, 3)}_{L_{q}(\Omega)}.
\end{equation*}
Since $(\widehat{u},\phi_0 )_{L_2(\Omega)} = 0,$ representation \eqref{seminorm_Sp} implies $\langle (-\Delta)_{Sp}^s \widehat{u}, \widehat{u} \rangle \geq \lambda_1^{s} \| \widehat{u} \|^2_{L_2(\Omega)};$ therefore the bounded embedding $\mathcal{H}^s(\Omega) \hookrightarrow L_q(\Omega)$ gives
\begin{multline}
\label{t2_p2_estimate}
\mathcal{J}^{\varepsilon}_{s, q}[u] 
\geq
\left(\lambda_1^s - (q -2) \varepsilon^{2s} \right) \lambda_1^{-s} \langle (-\Delta)_{Sp}^s \widehat{u}, \widehat{u} \rangle - C(q) \varepsilon^{2s} \| \widehat{u} \|^{\min(q, 3)}_{L_q(\Omega)}
\\ \geq
\left[\left(\lambda_1^s - (q -2) \varepsilon^{2s} \right) \lambda_1^{-s} - C(s,q) \varepsilon^{2s} \langle (-\Delta)_{Sp}^s \widehat{u}, \widehat{u} \rangle^{\min(q-2, 1)/2} \right] \langle (-\Delta)_{Sp}^s \widehat{u}, \widehat{u} \rangle.
\end{multline}
For a sufficiently small neighbourhood of the function $\mathbf{1},$ the right-hand side is positive (here $C(s,q)$ is the same as in \eqref{t2_p2_estimate}):
\begin{equation}
\label{t2_p2_constant}
\langle (-\Delta)_{Sp}^s \widehat{u}, \widehat{u} \rangle \leq \| u - \mathbf{1} \|^2_{\mathcal{H}^s(\Omega)} \leq 
\min \left[\frac{1}{2}, \left( \frac{\lambda_1^s - (q -2) \varepsilon^{2s}}{2 C(s,q) \lambda_1^{s} \varepsilon^{2s}} \right)^{2/\min(q-2, 1)} \right].
\end{equation}
In this neighborhood $\mathcal{J}^{\varepsilon}_{s, q}[u] > 0,$ and the statement follows.

\medskip

3. For $\varepsilon = \varepsilon_s(q)$ we have $\mathbf{D}^2 \mathcal{J}^{\varepsilon_s}_{s, q}[\mathbf{1}; \phi_1, \phi_1] = 0.$ We compute the third differential: one has $(\phi_1,\phi_0 )_{L_2(\Omega)} = 0,$ therefore 
\begin{equation*}
\mathbf{D}^3 \mathcal{J}^{\varepsilon_s}_{s, q}[\mathbf{1}; \phi_1, \phi_1, \phi_1]
= - 2 (q-1) (q-2) \varepsilon^{2s} \int\limits_{\Omega} \phi^3_1 \, dx \neq 0.
\end{equation*}
This means that the constant function $\mathbf{1}$ does not give a local minimum to $\mathcal{I}^{\varepsilon_s}_{s, q}[u],$ and therefore is not a global minimum either. The existence of $\breve{\varepsilon}(q)$ follows from the continuous dependence of the functional $\mathcal{I}^{\varepsilon_s}_{s, q}[u]$ on the parameter $\varepsilon.$
\end{proof}
\begin{remark}
In \cite{Per_Per} it was shown that the condition $\int\limits_{\Omega} \phi^3_1 \, dx \neq 0$ is fulfilled for a generic asymmetric domain.
\end{remark}
The following statement is a generalization of Proposition 3.4 in \cite{Naz_Sch}:
\begin{prop}
Let $s < n/2$ (this restriction is essential only for $n=1$), $q = 2^*_{s}$ and $\varepsilon = \varepsilon_s(2^*_{s}).$ Then the constant function $\mathbf{1}$ is not a global minimizer of the functional $\mathcal{I}^{\varepsilon_s}_{s, 2^*_{s}}[u]$ in the cube $(0,1)^n.$
\label{prop_cube}
\end{prop}
\begin{proof}
It is well-known that $\lambda_1 = \pi^2$ for the cube $(0,1)^n.$ Therefore
\begin{equation*}
\mathcal{I}^{\varepsilon_s}_{s, 2^*_{s}}[\mathbf{1}] = \varepsilon_s^{2s}(2^*_{s}) = \frac{\pi^{2s} (n-2s)}{4s}.
\end{equation*}
Let us construct a test sequence for the functional~$\mathcal{I}^{\varepsilon_s}_{s, 2^*_{s}}[u].$ Recall that the fractional Sobolev inequality holds in $\mathbb{R}^n$ (see \cite[Theorem 1.2, (22)]{Ilin}):
\begin{equation*}
\mathcal{S}_{s, \mathbb{R}^n} := \min \frac{\langle (-\Delta)^s u, u \rangle}{\| u \|^2_{L_{2^*_{s}}(\mathbb{R}^n)}} > 0
\end{equation*}
(here $(-\Delta)^s$ is the standard fractional Laplacian in $\mathbb{R}^n,$ and the minimum is taken over all functions that give finite numerator and denominator) and the minimum is attained (see \cite{Tav}) on the unique function (up to dilations, multiplications by a constant, and translations)
\begin{equation*}
\Phi_{s, a}(x) := \left(a^2 + |x|^2 \right)^{(2s-n)/2} \quad \mbox{with} \quad 
\mathcal{S}_{s, \mathbb{R}^n} = 2^{2s}\pi^{s} \frac{\Gamma\left(n/2+s \right)}{\Gamma\left(n/2 - s \right)} \left[ \frac{\Gamma\left(n/2\right)}{\Gamma\left( n \right)}  \right]^{2s/n}.
\end{equation*}
Using the arguments similar to the proof of the Lemma 4 in \cite{Ustinov_Neum}, one can obtain
\begin{equation*}
\lim_{a \to 0} \mathcal{I}^{\varepsilon_s, (0,1)^n}_{s, 2^*_{s}}[\Phi_{s, a}] = 2^{-2s} \lim_{a \to 0} \mathcal{I}^{\varepsilon_s, (-1,1)^n}_{s, 2^*_{s}}[\Phi_{s, a}] =  2^{-2s} \mathcal{S}_{s, \mathbb{R}^n},
\end{equation*}
and to complete the proof, it suffices to show that 
\begin{equation}
\label{eq_for_cube}
\frac{\pi^{2s} (n-2s)}{4s} > 2^{-2s} \mathcal{S}_{s, \mathbb{R}^n} = \pi^{s} \frac{\Gamma\left((n+2s)/2\right)}{\Gamma\left((n-2s)/2\right)} \left[ \frac{\Gamma\left(n/2\right)}{\Gamma\left( n \right)}  \right]^{2s/n}.
\end{equation}
This inequality is proved in \S 3.
\end{proof}
\begin{remark}
The functional $\mathcal{I}^{\varepsilon}_{s, q}[u]$ depends on the parameters $q$ and $\varepsilon$ continuously. Thus, if $q$ is close to $2^*_{s}$ from below and $\varepsilon$ is close to $\varepsilon_s(q)$ from below, then the constant function~$\mathbf{1}$ does not give a global minimum to $\mathcal{I}^{\varepsilon}_{s, q}[u]$ in the cube $(0,1)^n$ (although it gives a local minimum).
\end{remark}

Now we study the conditions for the functional $\mathcal{I}^{\varepsilon}_{s, q}[u]$ to have the constant function~$\mathbf{1}$ as a global minimizer.
\begin{theorem}
\label{const_min}
\begin{enumerate}
\item For any $q \in (2, 2^*_{s}]$ there exists $\check{\varepsilon}(q) > 0,$ such that for all $\varepsilon < \check{\varepsilon}(q)$ the constant function~$\mathbf{1}$ is a unique global minimizer of the functional $\mathcal{I}^{\varepsilon}_{s, q}[u].$ If $n = 1$ and $s \geq 1/2,$ then the statement remains valid for all $q \in (2, \infty).$
\item For any $\varepsilon > 0$ there exists $\check{q}(\varepsilon) \in (2, 2^*_{s}],$ such that for all $q < \check{q} (\varepsilon)$ the constant function~$\mathbf{1}$ is a unique global minimizer of the functional $\mathcal{I}^{\varepsilon}_{s, q}[u].$ If $n = 1$ and $s \geq 1/2,$ then the statement remains valid with $\check{q}(\varepsilon) < \infty.$
\item Let for the pair $(\varepsilon_0, q_0)$ the constant function $\mathbf{1}$ is not a global minimizer of the functional~$\mathcal{I}^{\varepsilon_0}_{s, q_0}[u].$ Then for any $\varepsilon \geq \varepsilon_0,$ $q \geq q_0$ the constant function $\mathbf{1}$ is not a global minimizer of the functional~$\mathcal{I}^{\varepsilon}_{s, q}[u]$ either.
\end{enumerate}
\end{theorem}
\begin{proof}
1. We argue by contradiction. Suppose that for some $q$ there exist a sequence $\check{\varepsilon}_n \downarrow 0$ and a sequence of non-constant functions $\{u_n \colon \|u_n\|_{L_q(\Omega)} = 1 \},$ such that
\begin{equation}
\label{ineq_un_q_epsn}
\mathcal{I}^{\check{\varepsilon}_n}_{s, q}[u_n] \leq \check{\varepsilon}^{2s}_n, \quad \mbox{therefore} \quad  \langle (-\Delta)_{Sp}^s u_n, u_n \rangle \leq \check{\varepsilon}^{2s}_n \quad \mbox{and} \quad \|u_n\|_{L_{2}(\Omega)} \leq 1.
\end{equation}
From \eqref{ineq_un_q_epsn} it follows that $u_n$ are bounded in $\mathcal{H}^s(\Omega).$ Passing to a subsequence, we can assume that $u_n \rightharpoondown u$ in~$\mathcal{H}^s(\Omega),$ and due to the compactness of embedding $\mathcal{H}^s(\Omega) \hookrightarrow L_2(\Omega)$ one has $u_n \to u$ in $L_{2}(\Omega).$ Since the norm in $\mathcal{H}^s(\Omega)$ is weakly lower semicontinuous, we obtain
\begin{equation*}
\langle (-\Delta)_{Sp}^s u, u \rangle + \|u\|^2_{L_{2}(\Omega)} = \| u \|^2_{\mathcal{H}^s(\Omega)} \leq \lim \limits_{n \to \infty} \| u_n \|^2_{\mathcal{H}^s(\Omega)} =  \|u\|^2_{L_{2}(\Omega)},
\end{equation*}
therefore $\langle (-\Delta)_{Sp}^s u, u \rangle = 0.$ Thus, $u$ is a constant function, and $u_n \to u$ in~$\mathcal{H}^s(\Omega).$ Since $\|u\|_{L_q(\Omega)} = 1,$ then $u = \mathbf{1}.$

Thus, in an arbitrarily small neighbourhood of the function $\mathbf{1}$ there exists a non-constant minimizer $u_n$ of the functional $\mathcal{I}^{\check{\varepsilon}_n}_{s, q}[u].$ For small $\check{\varepsilon}_n,$ the estimate \eqref{t2_p2_constant} of the neighbourhood of the function $\mathbf{1}$ is uniform in $n,$ and $u_n$ fall into this neighborhood. The inequallity \eqref{t2_p2_estimate} entails $\mathcal{J}^{\check{\varepsilon}_n}_{s, q}[u_n] > 0,$ which contradicts to the inequality \eqref{ineq_un_q_epsn}.

\medskip

2. We again argue by contradiction. Suppose that for some $\varepsilon > 0$ there exist a sequence $\check{q}_n \downarrow 2$ and a sequence of non-constant functions $\{u_n \colon \|u_n\|_{L_{\check{q}_n}(\Omega)} = 1 \},$ such that 
\begin{equation*}
\mathcal{I}^{\varepsilon}_{s, \check{q}_n}[u_n] \leq \varepsilon^{2s}, \quad \mbox{therefore} \quad  \langle (-\Delta)_{Sp}^s u_n, u_n \rangle \leq \varepsilon^{2s} \quad \mbox{and} \quad \|u_n\|_{L_{2}(\Omega)} \leq 1.
\end{equation*}
As in the proof of the previous statement, $u_n$ are bounded in $\mathcal{H}^s(\Omega)$ and we can assume that $u_n \rightharpoondown u$ in $\mathcal{H}^s(\Omega).$ We fix some $q \in (2, 2^*_{s});$ due to the compactness of the embedding $\mathcal{H}^s(\Omega) \hookrightarrow L_q(\Omega)$ one has $u_n \to u$ in $L_q(\Omega).$ Obviously, $\check{q}_n < q$ for large $n,$ and the H{\"o}lder inequality $\|u_n\|_{L_{q}(\Omega)} \geq \|u_n\|_{L_{\check{q}_n}(\Omega)}$ gives
\begin{equation}
\label{ineq_un_q_eps}
\mathcal{I}^{\varepsilon}_{s, q}[u_n] \leq \mathcal{I}^{\varepsilon}_{s, \check{q}_n}[u_n] \leq \varepsilon^{2s} = \mathcal{I}^{\varepsilon}_{s, \check{q}_n}[\mathbf{1}] = \mathcal{I}^{\varepsilon}_{s, q}[\mathbf{1}].
\end{equation}
From \eqref{ineq_un_q_eps} we obtain
\begin{equation*}
\lim_{n \to \infty} \langle (-\Delta)_{Sp}^s u_n, u_n \rangle \leq  \varepsilon^{2s} \lim_{n \to \infty} \|u_n\|^2_{L_q(\Omega)} - \varepsilon^{2s} \lim_{n \to \infty} \|u_n\|^2_{L_2(\Omega)} = \varepsilon^{2s} \left( \|u\|^2_{L_q(\Omega)} - \|u\|^2_{L_2(\Omega)} \right),
\end{equation*}
and, passing to the limit in $q \downarrow 2$ on the right-hand side, we get $\langle (-\Delta)_{Sp}^s u, u \rangle = 0.$ Therefore, as in the previous statement, $u_n \to u$ in $\mathcal{H}^s(\Omega)$ and $u = \mathbf{1}.$ 
By choosing $q$ close to 2 we can achieve $\varepsilon < \varepsilon_s(q),$ and in this case the inequality~\eqref{ineq_un_q_eps} contradicts to the inequality \eqref{t2_p2_estimate}.

\medskip

3. Since the constant function $\mathbf{1}$ is not a global minimizer of the functional $\mathcal{I}^{\varepsilon_0}_{s, q_0}[u],$ there exists a non-constant function $u$ such that $\mathcal{I}^{\varepsilon_0}_{s, q_0}[u] < \varepsilon_0^{2s}.$ The required statement follows from the H{\"o}lder inequality:
\begin{gather*}
\mathcal{I}^{\varepsilon_0}_{s, q}[u] \leq \mathcal{I}^{\varepsilon_0}_{s, q_0}[u] < \varepsilon_0^{2s} = \mathcal{I}^{\varepsilon_0}_{s, q}[\mathbf{1}], \\
\mathcal{I}^{\varepsilon}_{s, q}[u] = \mathcal{I}^{\varepsilon_0}_{s, q}[u] + (\varepsilon^{2s} - \varepsilon_0^{2s})  \tfrac{\| u \|^2_{L_2(\Omega)}}{\| u \|^2_{L_q(\Omega)}} < \varepsilon_0^{2s} + (\varepsilon^{2s} - \varepsilon_0^{2s}) = \mathcal{I}^{\varepsilon}_{s, q}[\mathbf{1}].
\qedhere
\end{gather*}
\end{proof}

Based on Theorem \ref{const_min}, we can determine the function $\mathcal{E}_s(q),$ $q \in (2, 2^*_{s}]:$ for $\varepsilon \leq \mathcal{E}_s(q)$ the constant function $\mathbf{1}$ gives a global minimum to the functional $\mathcal{I}^{\varepsilon}_{s, q}[u],$ and for $\varepsilon > \mathcal{E}_s(q)$ the constant function~$\mathbf{1}$ does not give a global minimum to the functional $\mathcal{I}^{\varepsilon}_{s, q}[u].$ Obviously, $\varepsilon_s(q) \geq \mathcal{E}_s(q).$ Moreover, Theorem \ref{const_min} gives that $\mathcal{E}_s(q) > 0,$ $\mathcal{E}_s(q)$ is a non-increasing function, and $\mathcal{E}_s(q) \to \infty$ for $q \downarrow 2.$ Finally, if $n = 1$ and $s \geq \frac{1}{2},$ then $\mathcal{E}_s(q) \to 0$ for $q \to \infty.$

\begin{theorem}
\label{cont_epsilon_th}
The function $\mathcal{E}_s(q)$ is continuous on $(2, 2^*_{s}]$ (on $(2, \infty)$ in case $n = 1,$ $s \geq 1/2$) and strictly decreases.
\end{theorem}
\begin{proof}
We prove the statement of Theorem for $s < n/2.$ For the case $n = 1,$ $s \geq 1/2$ the proof goes without changes.

\medskip

1. First, we show that $\mathcal{E}_s(q)$ is strictly decreasing. Indeed, consider the point $(q_0, \varepsilon_0)$ on the curve $\varepsilon = \mathcal{E}_s(q),$  $q_0 \in (2, 2^*_{s})$. By definition of $\mathcal{E}_s(q),$ for any $\varepsilon_n \downarrow \varepsilon_0$ the minimizers~$u_n$ of functionals~$\mathcal{I}^{\varepsilon_n}_{s, q_0}[u]$ are not constant. We normalize these minimizers with the conditions $\|u_n\|_{L_{q_0}(\Omega)} = 1.$ The following inequalities hold:
\begin{equation}
\label{monot_proof_ineq}
\mathcal{I}^{\varepsilon_n}_{s, q_0}[u_n] < \varepsilon^{2s}_n, \quad \mbox{therefore} \quad  \langle (-\Delta)_{Sp}^s u_n, u_n \rangle < \varepsilon^{2s}_n \quad \mbox{and} \quad \|u_n\|_{L_{2}(\Omega)} < 1.
\end{equation}
As before, $u_n$ are bounded in $\mathcal{H}^s(\Omega),$ there exists a weak limit $u_n \rightharpoondown u_0$ in~$\mathcal{H}^s(\Omega)$ and $u_n \to u_0$ in $L_{q_0}(\Omega)$ and in $L_{2}(\Omega).$ From the weak lower semicontinuity of the norm in $\mathcal{H}^s(\Omega),$ it follows that
\begin{equation*}
\mathcal{I}^{\varepsilon_0}_{s, q_0}[u_0] \leq  \lim\limits_{n \to \infty}\mathcal{I}^{\varepsilon_n}_{s, q_0}[u_n] \leq \lim\limits_{n \to \infty}\varepsilon^{2s}_n = \varepsilon^{2s}_0.
\end{equation*}
There are two possible cases: $u_0 \neq \mathbf{1}$ and $u_0 = \mathbf{1}.$

In the first case, for $q > q_0$ we have $\|u_0\|_{L_q(\Omega)} > \|u_0\|_{L_{q_0}(\Omega)} = 1;$ therefore
\begin{equation*}
\mathcal{I}^{\varepsilon_0}_{s, q}[u_0] < \mathcal{I}^{\varepsilon_0}_{s, q_0}[u_0] \leq \varepsilon^{2s}_0 = \mathcal{I}^{\varepsilon_0}_{s, q}[\mathbf{1}].
\end{equation*}
This means that for $u_0 \neq \mathbf{1}$ and $q > q_0$ the constant function $\mathbf{1}$ does not give a global minimum for the functional $\mathcal{I}^{\varepsilon_0}_{s, q}[u],$ therefore $\mathcal{E}_s(q) < \varepsilon_0 = \mathcal{E}_s(q_0).$

We claim that in the second case the inequality $\varepsilon_0 < \varepsilon_s(q_0)$ is impossible. Indeed, assume the converse. Then, for large~$n$ we have $\varepsilon_0 < \varepsilon_n < \varepsilon_s(q_0),$ and, according to Theorem \ref{local_min}, the constant function $\mathbf{1}$ is a local minimizer of the functional $\mathcal{I}^{\varepsilon_n}_{s, q_0}[u].$ Moreover, 
\begin{equation*}
\lim_{n \to \infty} \langle (-\Delta)_{Sp}^s u_n, u_n \rangle \leq \lim_{n \to \infty}  \varepsilon_n^{2s} \left( \|u_n\|^2_{L_{q_0}(\Omega)} - \|u_n\|^2_{L_2(\Omega)} \right) = \varepsilon_0^{2s} \left( \|\mathbf{1}\|^2_{L_{q_0}(\Omega)} - \|\mathbf{1}\|^2_{L_2(\Omega)} \right) = 0;
\end{equation*}
therefore $\langle (-\Delta)_{Sp}^s u, u \rangle = 0$ and $u_n \to \mathbf{1}$ in $\mathcal{H}^s(\Omega).$ Together with \eqref{monot_proof_ineq} this gives a contradiction, similarly to the proof of Theorem  \ref{const_min}.

Since $\varepsilon_0 = \mathcal{E}_s(q_0) \leq \varepsilon_s(q_0),$ the only possibility is $\varepsilon_0 = \varepsilon_s(q_0),$ and the strict monotonicity follows from the inequalities
\begin{equation*}
\mathcal{E}_s(q) \leq \varepsilon_s(q) < \varepsilon_s(q_0) =  \varepsilon_0 = \mathcal{E}_s(q_0).
\end{equation*}

\medskip

2. As a second step, we show that the monotone function $\mathcal{E}_s(q)$ is left continuous. To do this, consider again a point $(q_0, \varepsilon_0)$ on the curve $\varepsilon = \mathcal{E}_s(q),$ $q_0 \in (2, 2^*_{s}]:$ for~$\varepsilon > \varepsilon_0$ the constant function $\mathbf{1}$ is not a global minimizer of the functional $\mathcal{I}^{\varepsilon}_{s, q_0}[u].$ By the continuity of $L_q$-norm with respect to $q$ sufficiently close to $q_0$ from below, we see that the constant function $\mathbf{1}$ is not a global minimizer of the functional $\mathcal{I}^{\varepsilon}_{s, q}[u].$ This gives the continuity of $\mathcal{E}_s(q)$ on the left.

\medskip

3. It remains to show the right continuity of the function $\mathcal{E}_s(q).$ Let $q_0 \in (2, 2^*_{s})$ and let $\varepsilon_0 = \lim \limits_{q \downarrow q_0}  \mathcal{E}_s(q).$ Due to the monotonicity of the function $\mathcal{E}_s(q),$ for any $q_n \downarrow q_0$ minimizers~$u_n$ of the functionals~$\mathcal{I}^{\varepsilon_0}_{s, q_n}[u]$ are not constant. As before, we normalize these minimizers with the condition $\|u_n\|_{L_{q_n}(\Omega)} = 1.$ The following inequalities hold:
\begin{equation*}
\mathcal{I}^{\varepsilon_0}_{s, q_n}[u_n] < \varepsilon_0^{2s}, \quad \mbox{therefore} \quad  \langle (-\Delta)_{Sp}^s u_n, u_n \rangle < \varepsilon_0^{2s} \quad \mbox{and} \quad \|u_n\|_{L_{2}(\Omega)} \leq 1.
\end{equation*}
We repeat the reasoning from the proof of the second part of Theorem \ref{const_min}: $u_n$ are bounded in~$\mathcal{H}^s(\Omega)$ and there exists a weak limit $u_n \rightharpoondown u_0$ in~$\mathcal{H}^s(\Omega).$ For any fixed $q < 2^*_{s}$ one has $u_n \to u_0$ in $L_q(\Omega),$ and the H{\"o}lder inequality for $q_0 < q_n \leq q$ gives $\|u_n\|_{L_{q}(\Omega)} \geq \|u_n\|_{L_{q_n}(\Omega)} = 1;$ therefore
\begin{multline*}
\langle (-\Delta)_{Sp}^s u_0, u_0 \rangle 
\leq 
\lim_{n \to \infty} \langle (-\Delta)_{Sp}^s u_n, u_n \rangle
\leq
\varepsilon_0^{2s} \lim_{n \to \infty} \|u_n\|^2_{L_q(\Omega)} - \varepsilon_0^{2s} \lim_{n \to \infty} \|u_n\|^2_{L_2(\Omega)}
\\=
\varepsilon_0^{2s} \left( \|u_0\|^2_{L_q(\Omega)} - \|u_0\|^2_{L_2(\Omega)} \right).
\end{multline*}
Passing to the limit in $q\downarrow q_0$ we obtain $\mathcal{I}^{\varepsilon_0}_{s, q_0}[u_0] \leq \varepsilon_0^{2s}.$ Moreover, since 
\begin{gather*}
\|u_0\|_{L_{q}(\Omega)} = \lim_{n \to \infty} \|u_n\|_{L_{q}(\Omega)} \geq \lim_{n \to \infty} \|u_n\|_{L_{q_n}(\Omega)} = 1, \\
\|u_0\|_{L_{q_0}(\Omega)} = \lim_{n \to \infty} \|u_n\|_{L_{q_0}(\Omega)} \leq \lim_{n \to \infty} \|u_n\|_{L_{q_n}(\Omega)} = 1,
\end{gather*}
passing to the limit in $q \downarrow q_0$ we obtain $\|u_0\|_{L_{q_0}(\Omega)} = 1.$ Here, as in the proof of the strict monotonicity of the function $\mathcal{E}_s(q),$ there are two cases: $u_0 \neq \mathbf{1}$ and $u_0 = \mathbf{1}.$

In the first case, for any $\varepsilon > \varepsilon_0$ we get from the estimate $\|u_0\|_{L_{2}(\Omega)} < \|u_0\|_{L_{q_0}(\Omega)} = 1$
\begin{equation*}
\mathcal{I}^{\varepsilon}_{s, q_0}[u_0] = \mathcal{I}^{\varepsilon_0}_{s, q_0}[u_0] + (\varepsilon^{2s} - \varepsilon_0^{2s}) \frac{\| u_0 \|^2_{L_2(\Omega)}}{\| u_0 \|^2_{L_{q_0}(\Omega)}} < \varepsilon_0^{2s} + (\varepsilon^{2s} - \varepsilon_0^{2s}) = \mathcal{I}^{\varepsilon}_{s, q_0}[\mathbf{1}],
\end{equation*}
which gives $\varepsilon_0 \geq \mathcal{E}_s(q_0).$ Since the function $\mathcal{E}_s(q)$ decreases, then $\varepsilon_0 \leq \mathcal{E}_s(q_0)$ and therefore $\varepsilon_0 = \mathcal{E}_s(q_0).$

In the second case $u_n \to \mathbf{1}$ in $\mathcal{H}^s(\Omega),$ and for $q > q_0$ from $\|u_n\|_{L_{q}(\Omega)} \geq \|u_n\|_{L_{q_n}(\Omega)}$ we get
\begin{equation*}
\mathcal{I}^{\varepsilon_0}_{s, q}[u_n] \leq \mathcal{I}^{\varepsilon_0}_{s, q_n}[u_n] < \varepsilon_0^{2s} = \mathcal{I}^{\varepsilon_0}_{s, q_n}[\mathbf{1}] = \mathcal{I}^{\varepsilon_0}_{s, q}[\mathbf{1}],
\end{equation*}
which gives $\varepsilon_0 \geq \varepsilon_s(q).$ Passing to the limit in $q\downarrow q_0$ we obtain $\varepsilon_0 \geq \varepsilon_s(q_0) \geq \mathcal{E}_s(q_0).$ Analogously to the first case we get 
$\varepsilon_0 = \mathcal{E}_s(q_0),$ since the function $\mathcal{E}_s(q)$ decreases.
\end{proof}
\begin{remark}
In \cite{Naz2} it was shown that for $n=1,$ $s \geq 1$ the equality $\varepsilon_s(q) = \mathcal{E}_s(q)$ holds for all $q > 2,$ and it was also conjectured that this statement remains valid for all $s \geq 1/2.$ This is still an open problem, but Proposition \ref{prop_cube} shows that for $s < 1/2$ the equality $\varepsilon_s(q) = \mathcal{E}_s(q)$ does not hold for $q,$ close to $2^*_{s}$ from below.
\end{remark}

\section{Proof of the inequality \eqref{eq_for_cube}}
\begin{proof}
Recall that $n>2s$ and rewrite the inequality \eqref{eq_for_cube} as
\begin{equation}
\label{main_cube_ineq}
\mathcal{A}_n := \frac{2s \Gamma\left((n+2s)/2\right)}{\Gamma\left((n-2s+2)/2\right)} \left[ \frac{\Gamma\left(n/2\right)}{\Gamma\left( n \right)}  \right]^{2s/n} < \pi^{s}.
\end{equation}
Let us prove that $\mathcal{A}_{n+2} < \mathcal{A}_{n},$ i.e.
\begin{equation*}
\frac{\mathcal{A}_{n+2}}{\mathcal{A}_n}
 = 
\frac{n+2s}{n-2s+2} \left[ \frac{\Gamma\left( n \right)}{\Gamma\left( n/2 \right) \left[ 2 (n+1)\right]^{n/2} }  \right]^{4s/(n(n+2))} < 1.
\end{equation*}
By raising to a power, we obtain the equivalent inequality
\begin{equation}
\label{arb_s_cube_ineq}
\left[\frac{n+2s}{n-2s+2}\right]^{n(n+2)/(4s)} \frac{\Gamma\left( n \right)}{\Gamma\left( n/2 \right)  \left[ 2 (n+1)\right]^{n/2}} < 1.
\end{equation}
Let us show that the function $f(s) = \left[\tfrac{n+2s}{n-2s+2}\right]^{n(n+2)/(4s)}$ monotonically increases as $s \in [0,1].$ Its logarithmic derivative is equal to
\begin{multline*}
\frac{d}{ds}\left[ \frac{1}{s} \ln\left(\frac{n+2s}{n-2s+2}\right) \right] 
=
-\frac{1}{s^2} \ln\left(\frac{n+2s}{n-2s+2}\right) + \frac{2(n-2s+2)(2n+2)}{s(n+2s)(n-2s+2)^2} 
\\=
\frac{1}{s} \left[\frac{4n+4}{(n+2s)(n-2s+2)} - \frac{1}{s} \ln\left(\frac{n+2s}{n-2s+2}\right)\right]
\\ \geq
\frac{1}{s} \left[\frac{4n+4}{(n+2s)(n-2s+2)} - \frac{4s - 2}{s (n-2s+2)}\right]
=
\frac{2}{s(n-2s+2)} \left[\frac{2n+2}{n+2s} - \frac{2s - 1}{s}\right]
\\=
\frac{2n + 8s -  8s^2}{s^2(n-2s+2)(n+2s)} \geq 0.
\end{multline*}
Thus, it suffices to prove the inequality \eqref{arb_s_cube_ineq} for $s=1,$ i.e. the inequality
\begin{equation*}
\frac{\Gamma\left( n/2 \right)}{\Gamma\left( n \right)}  \left[ 2 (n+1)  \left(\frac{n}{n+2}\right)^{(n+2)/2}  \right]^{n/2} > 1.
\end{equation*}
For $n = 1$ this inequality is obvious. For $n \geq 2$ we prove a stronger inequality
\begin{equation*}
\frac{\Gamma\left( n/2 \right)}{\Gamma\left( n \right)} \left[ 2 (n+1)  \left(\frac{n}{n+2}\right)^{(n+2)/2}  \right]^{n/2} \geq 1 + \frac{2}{n+2}
\end{equation*}
which is equivalent to the inequality
\begin{equation}
\mathcal{B}_n := \frac{ (n+2) \Gamma\left( n/2 \right)}{ (n+4) \Gamma\left( n \right)} \left[ 2 (n+1)  \left(\frac{n}{n+2}\right)^{(n+2)/2}  \right]^{ n/2} \geq 1.
\label{final_cube_ineq}
\end{equation}
Direct calculations show that $\mathcal{B}_2 = 1$ and $\mathcal{B}_3 \geq 1,05.$  We compute the ratio $\mathcal{B}_{n+2} / \mathcal{B}_n:$
\begin{multline*}
\frac{\mathcal{B}_{n+2}}{\mathcal{B}_n}
=
\frac{ (n+4)^2}{2 (n+1) (n+2) (n+6)}  \left[ 2 (n+3)  \left(\frac{n+2}{n+4}\right)^{(n+4)/2}  \right]^{(n+2)/2}
\\ \times
\left[ 2 (n+1)  \left(\frac{n}{n+2}\right)^{(n+2)/2}  \right]^{-n/2}
\\=
\frac{ (n+4)^2}{(n+2) (n+6)}  
\left[ \frac{(n+3)(n+2)^2}{(n+1)(n+4)^2} \left(\frac{(n+2)^2}{n^2+4n}\right)^{n/2} \right]^{(n+2)/2} =: g(n).
\end{multline*}
Let us show that the function $g(x)$ decreases monotonically as $x \geq 2.$ Its logarithmic derivative is equal to
\begin{multline*}
\mathcal{D} := \frac{d}{dx}\left[ 
\ln \left(\frac{ (x+4)^2}{(x+2) (x+6)} \right) + \frac{x+2}{2} \bigl[
\ln \left(\frac{x+3}{x+1} \right) - 2 \ln \left(\frac{x+4}{x+2}\right) + \frac{x}{2} \ln \bigl(\frac{(x+2)^2}{x^2+4x}\bigl)
\bigr] \right]
\\=
-\frac{8}{(x+2) (x+4) (x+6)}
+
\frac{1}{2} \left[ \ln \left(1 + \frac{2}{(x+1)(x+4)} \right) - \ln \left( 1 + \frac{2}{x+2}\right) \right]
\\+
\frac{x+1}{2} \ln \left(1 + \frac{4}{x^2+4x}\right)
-
\frac{x+2}{(x+1)(x+3)} .
\end{multline*}
Since for $t \in (0, 1)$ one has the estimates
\begin{equation*}
t > t - \frac{t^2}{2} + \frac{t^3}{3}  > \ln(1 + t) > t - \frac{t^2}{2},
\end{equation*}
for $x \geq 2$ we have (put $y := x - 2$)
\begin{multline*}
\mathcal{D} \leq
- \frac{8}{(x+2) (x+4) (x+6)} + \frac{1}{(x+1)(x+4)} - \frac{x+1}{(x+2)^2}
\\+ \frac{x+1}{2}   \left( \frac{4}{x^2+4x} - \frac{8}{(x^2+4x)^2} +  \frac{64}{3(x^2+4x)^3} \right) 
- \frac{x+2}{(x+1)(x+3)}
\\=
-\frac{18x^8 + 219x^7+ 910x^6 + 1236x^5 - 968x^4 - 4080x^3 - 5024x^2 - 4608x - 2304}{3x^3(x+1)(x+2)^2(x+3)(x+4)^3(x+6)}
\\=
-\frac{18y^8 + 507y^7+ 5992y^6 + 38616y^5 + 147472y^4 + 338112y^3 + 443968y^2 + 285504y + 50688}{3(y+5)^3(y+6)(y+7)^2(y+8)(y+9)^3(y+10)}
\\ < 0.
\end{multline*}
From the Bernoulli inequality $(1 + t)^m \geq 1 + mt$ for $t > -1,$ we obtain
\begin{gather*}
\left(\frac{(n+2)^2}{n^2+4n}\right)^{n/2} = 
\left(1 + \frac{4}{n^2+4n}\right)^{n/2} \geq \frac{n+6}{n+4}; \\
\left[ \frac{(n+3)(n+2)^2(n+6)}{(n+1)(n+4)^3} \right]^{(n+2)/2} 
= 
\left[1 - \frac{2(n^2 + 2n -4)}{(n+1)(n+4)^3}  \right]^{(n+2)/2}
\geq
1 - \frac{(n^2 + 2n -4)(n+2)}{(n+1)(n+4)^3},
\end{gather*}
which gives 
\begin{multline*}
\lim\limits_{n \to \infty} g(n) 
\geq 
\lim\limits_{n \to \infty} \frac{ (n+4)^2}{(n+2) (n+6)}  
\left[ \frac{(n+3)(n+2)^2(n+6)}{(n+1)(n+4)^3} \right]^{(n+2)/2}
\\ \geq 
\lim\limits_{n \to \infty} \left[1 - \frac{(n^2 + 2n -4)(n+2)}{(n+1)(n+4)^3} \right] = 1.
\end{multline*}
Thus, the function $g(n)$ decreases monotonically to one and therefore we have $g(n) >1.$ This proves the inequality \eqref{final_cube_ineq}, which implies the inequality \eqref{arb_s_cube_ineq}, and, therefore $\mathcal{A}_{n+2} < \mathcal{A}_{n}.$ 

To complete the proof, it remains to show that \eqref{main_cube_ineq} is valid for $n \leq 3$ (in the case $n = 1$ here is an additional constraint $2s < 1,$ thus \eqref{main_cube_ineq} for $n=3,$ $2s \geq 1$ should be checked separately). For $n = 1$ the inequality \eqref{main_cube_ineq} takes the form (since $\Gamma\left( 1/2 \right) = \sqrt{\pi}$)
\begin{equation*}
\frac{2s  \Gamma\left((1+2s)/2\right)}{\Gamma\left( (3-2s)/2 \right)} < 1.
\end{equation*}
This inequality follows from the estimate $\Gamma\left( (3+2s)/2\right) \leq \Gamma\left(2\right) \leq \Gamma\left( (5-2s)/2 \right):$
\begin{equation*}
\frac{2s  \Gamma\left((1+2s)/2\right)}{\Gamma\left((3-2s)/2\right)}
=
\frac{2s (3 - 2s)  \Gamma\left((3+2s)/2\right)}{(1 + 2s)  \Gamma\left((5-2s)/2\right)} 
=
\left(1 - \frac{(1-2s)^2}{1 + 2s} \right)  \frac{\Gamma\left((3+2s)/2\right)}{\Gamma\left((5-2s)/2\right)}
<
 1.
\end{equation*}
For $n = 2$ the inequality \eqref{main_cube_ineq} is written as
\begin{equation*}
\frac{2s  \Gamma\left(1+s\right)}{\Gamma\left(2-s\right)}  < \pi^{s}.
\end{equation*}
Using the Euler reflection formula $\Gamma\left(z\right) \Gamma\left(1-z\right) = \pi/\sin(\pi z)$ for non-integer $z,$ we obtain
\begin{equation*}
2 \Gamma^2\left(1+s\right)  \sin(\pi s) < \pi^{1+s}  (1-s),
\end{equation*}
and this inequality follows from the following estimates (since $\Gamma\left(1+s\right) \leq 1$):
\begin{gather*}
2 \Gamma^2\left(1+s\right)  \sin(\pi s) < 2 < \pi^{1+s}  (1-s) \quad \mbox{for $s \in (0, \, 0,7]$}; \\
2 \Gamma^2\left(1+s\right)  \sin(\pi s) < 2 \sin(\pi(1-s)) < 2 \pi  (1-s) < \pi^{1+s}  (1-s) \quad \mbox{for $s \in [0,7,\, 1)$}.
\end{gather*}
For $n=3$ the inequality (\ref{main_cube_ineq}) is written as
\begin{equation*}
\frac{2s  \Gamma\left( (3+2s)/2 \right)}{\Gamma\left( (5-2s)/2 \right)} \left[ \frac{\sqrt{\pi}}{4}  \right]^{2s/3} < \pi^{s},
\end{equation*}
or, equivalently
\begin{equation*}
2s  \Gamma\left((3+2s)/2\right) < (4\pi)^{2s/3} \Gamma\left( (5-2s)/2 \right).
\end{equation*}
The validity of this inequality follows from the estimate
\begin{equation*}
2s  \Gamma\left((3+2s)/2\right) < 2s  \Gamma\left( 5/2 \right) = 3s  \Gamma\left( 3/2 \right) < (4\pi)^{2s/3}  \Gamma\left(3/2\right) < (4\pi)^{2s/3}  \Gamma\left( (5-2s)/2 \right). \qedhere
\end{equation*}
\end{proof}

The author is very grateful to A.I. Nazarov for posing the problem and for valuable discussions of the results, and also to A.P. Scheglova for useful remarks that made it possible to improve the text of the work.

\end{document}